\documentclass{amsproc}

\usepackage{}

\usepackage[centering,left=3.3cm,right=3.3cm]{geometry}
\usepackage{url}
\usepackage{mathrsfs}
\usepackage{enumitem}
\usepackage{color} 
\usepackage[colorlinks,linkcolor=blue,anchorcolor=blue,citecolor=blue,backref=page]{hyperref}
\usepackage{graphics,epsfig}
\usepackage{graphicx} 
\usepackage{float}
\usepackage{epstopdf}
\usepackage[utf8]{inputenc}
\usepackage{hyperref} 
\usepackage{calc, amscd}
\usepackage[T1]{fontenc} 
\usepackage{pdfpages}
\usepackage{mathtools} 
\usepackage{scalerel,stackengine}
\usepackage{caption}
\usepackage{mathptmx}
\usepackage[per-mode=symbol]{siunitx}
\usepackage{bm}
\usepackage{amsmath}
\hypersetup{breaklinks=true}
\usepackage{amsthm,mathrsfs, multirow,xcolor,lscape,longtable,pbox,lipsum}
\usepackage{amsfonts,amscd,amssymb}
\usepackage{tikz-cd}
\usepackage[thinlines]{easytable}
\usepackage{microtype}
\usepackage{todonotes}
\usepackage{phaistos}
\usepackage{tabularx}

\usepackage[OT2,T1]{fontenc}
\DeclareSymbolFont{cyrletters}{OT2}{wncyr}{m}{n}
\DeclareMathSymbol{\Sha}{\mathalpha}{cyrletters}{"58}

\newtheorem{theorem}{Theorem}[section]
\newtheorem{lemma}[theorem]{Lemma}

\theoremstyle{definition}
\newtheorem{definition}[theorem]{Definition}

\theoremstyle{remark}
\newtheorem{remark}[theorem]{Remark}

\newtheorem{corollary}{Corollary}

\newtheorem{Conjecture}{Conjecture}
\newtheorem{proposition}{Proposition}
\newtheorem{question}{Question}

\newtheorem*{theorem*}{Theorem}
\newtheorem*{question*}{Question}

\numberwithin{equation}{section}

\newcommand\T{\mathbb{T}}

\newcommand\GL{{\mathrm{GL}}}

\newcommand{\Q}{\mathbb{Q}}

\newcommand{\Z}{\mathbb{Z}}

\newcommand\Gal{{\mathrm {Gal}}}

\usepackage{mathtools}

\newcommand{\fn}{{\mathfrak n}}

\newcommand{\cW}{{\mathcal W}}

\DeclareMathOperator{\rank}{rank}

\newcommand\remove[1]{}

\newcommand{\subham}[2][]{\todo[#1,color=green!60]{SB: #2}}

\begin{document}

\title[Modular degree and a conjecture of Watkins]{Modular degree and a conjecture of Watkins}



\date{\today}

\author{Subham Bhakta}
\address{Indian Institute of Science Education and Research, Thiruvananthapuram, India}
\curraddr{School of Mathematics, University of New South Wales, Sydney, Australia}
\email{subham.bhakta@unsw.edu.au}
\thanks{SB is supported by an institute fellowship from IISER Thiruvananthapuram and ARC grant DP230100530}

\author{Srilakshmi Krishnamoorthy}
\address{Indian Institute of Science Education and Research, Thiruvananthapuram, India}
\curraddr{}
\email{srilakshmi@iisertvm.ac.in}
\thanks{SK is supported by SERB grant CRG/2023/009035.}

\author{Sunil Kumar Pasupulati}
\address{Indian Institute of Science Education and Research, Thiruvananthapuram, India}
\curraddr{The Institute of Mathematical Sciences,
Chennai, India.}
\email{sunilkp@imsc.ac.in, sunilkumarpasupulati@gmail.com}
\thanks{SKP is supported by an institute fellowship from IISER Thiruvananthapuram and IMSc Chennai.}



\begin{abstract}
   Given an elliptic curve $E/\Q$ of conductor $N$, there exists a surjective morphism $\phi_E: X_0(N) \to E$ defined over $\Q$. In this article, we discuss the growth of $\mathrm{deg}(\phi_E)$ and shed some light on Watkins's conjecture, which predicts $2^{\mathrm{rank}(E(\Q))} \mid \mathrm{deg}(\phi_E)$. Additionally, we investigate a specific family of elliptic curves over $\mathbb{F}_q(T)$, which are linked through an analogous modular parametrization to Drinfeld modular curves. In this case, we also discuss growth and the divisibility properties.
\end{abstract}

\maketitle

\section{Introduction}
\noindent
Let $X_0(N)$ be the moduli space of the pairs $\{(E,C_N)\mid C_N\subseteq E,~\mathrm{cyclic~subgroup~of~order}~N\}$. It is well known that $X_0(N)(\mathbb{C})$ can be realized as the Riemann surface obtained by the standard action of the congruence subgroup $\Gamma_0(N)$ on the upper half plane $\mathbb{H}$. We know that $E$ is modular by the work of Wiles, Taylor-Wiles, et al. In other words, there exists a surjective morphism (defined over $\Q$)
$\phi_{E}:X_0(N) \to E$, which is called the \textit{modular parametrization} of $E$. Throughout the article, we shall assume that $\phi_E$ has a minimal degree sending the cusp $\infty$ of $X_0(N)$ to the identity of $E$. We denote the minimal degree as $m_E$, the modular degree of $E$. The main goal of this article is to study growth and some divisibility properties of $m_E$. To be more precise, in Section~\ref{sec:bounds}, we discuss the bounds of $m_E$ in terms of the conductor $N_E$ of $E$. It is conjectured that $m_E\ll N_E^{2+\varepsilon}$, and this is equivalent \cite{Murty99} to the  ABC conjecture. We discuss this briefly in Section~\ref{sec:bounds}. In the same section, we record the conjectural bound for a positive proportion of elliptic curves. More precisely, we prove the following result.

\begin{theorem}
Let $S$ be any finite set of primes, then the proportion of elliptic curves that satisfy the degree bound $m_E\ll_{S} N_E^{2+\varepsilon}$, is at least $1-\sum_{p\not\in S}\frac{1}{p^2}$.
\end{theorem}

Another main focus of this article is to examine the following prediction by Watkins.
\begin{Conjecture}[Watkins]
    Let $E/\Q$ be any elliptic curve and $m_E$ be the modular degree. Then, $$2^{\mathrm{rank}_{\Z} \left(E(\Q)\right)} \mid m_E.$$
    In other words, $\mathrm{rank}_{\Z}(E(\Q))\leq \nu_2(m_E)$.
\end{Conjecture}
It is not known whether the rank of elliptic curves is uniformly bounded, \footnote{
    In 2006, Noam Elkies discovered an elliptic curve with a rank of at least $28$, 
    \begin{align*}
        y^2 +& xy + y = x^3 - x^2 -20067762415575526585033208209338542750930230312178956502x + \\ &34481611795030556467032985690390720374855944359319180361266008296291939448732243429
    \end{align*}
At present, this is the highest known rank. For further information, please refer to  \href{https://web.math.pmf.unizg.hr/~duje/tors/rankhist.html}{A. Dujella. History of elliptic curves rank records. 2015.}  
} therefore the conjecture suggests that $\nu_2(m_E)$ could be as large as possible. This is indeed confirmed by Dummigan-Krishnamoorthy \cite{DK13}, since $\omega(N_E)-3$ is unbounded.

Dummigan \cite{neil06} provided a heuristic proof of Watkins's conjecture, using a certain Selmer group of the symmetric square of $E$
and its relation with the tangent space of a suitable deformation ring that captures relevant properties of the residual representation $\overline{\rho}:\mathrm{Gal}(\overline{\Q}/\Q)\to \mathrm{GL}_2(E[2])$. In the case when $m_E$ 
is odd, Watkins’s conjecture predicts that $E(\Q)$ is finite. By studying the Atkin-Lehner involution on elliptic curves of odd modular degrees, they showed that such curves have an even analytic rank, and there are at most two odd primes dividing the corresponding conductors. We shall briefly discuss this in Section~\ref{sec:odddeg}. Gross~\cite{Gross85} provided a way to relate $m_E$ with an element of the Picard group of $X_0(1) \pmod p$ when the elliptic curve has prime conductor $p$. Using this association, Kazalicki-Kohen \cite{KK18watkinscorrigendum} proved the conjecture for the odd modular degree case. They also showed that the conjecture is true for elliptic curves with the prime conductor of rank $1$, under a conjecture (see Conjecture~\ref{conj:supsing}) on the supersingular $j$-invariants of the elliptic curves over $\mathbb{F}_p$. We briefly discuss this in Section~\ref{sec:gross}, Section~\ref{sec:primecond}, and record Theorem~\ref{thm:supsingconj} as our take on the conjecture.

When we discuss even $m_E$, it is essential to note that Dummigan-Krishnamoorthy \cite{DK13} established a lower bound on $\nu_2(m_E)$ in terms of $\omega(N_E)$. In fact, it turns out that $\nu_2(m_E)\geq \omega(N_E)-1$ for almost all elliptic curves over $\Q$. Using this and Bhargava-Shankar's \cite{BS13} result on the average rank bound, we record another heuristic argument in Theorem~\ref{thm:avg} towards Watkins's conjecture, which says the following.
\begin{theorem}
    The average rank of all elliptic curves is bounded by the average of all $2$-adic valuations of modular degrees.
\end{theorem} 
When we allow $E$ to have a non-trivial 2-torsion $\Q$ point, one can show that $\mathrm{rank}_{\Z}(E(\Q))\leq \omega(N_E)-1$, when $E$ is semistable. In particular, when $E(\Q)[2]=\Z/2\Z$, Caro-Pasten proved Watkins's conjecture in \cite{CP22}, when the number of non-split multiplicative reduction is either odd or if there is no prime of split multiplicative reduction. Following the same approach, we prove an extension of this in Theorem~\ref{thm:CP22extension} for the case $E(\Q)[2]=(\Z/2\Z)^2$.

Given an elliptic curve $E/\Q$, we consider the family of quadratic twists $\{E^{(D)}\}$. Watkins's conjecture has also been studied for such families. In the case when $E$ has a non-trivial 2-torsion $\Q$ point, Esparza-Lozano and Pasten showed in \cite{EP21} that Watkins's conjecture is true when $\omega(D)$ is sufficiently large. The crux of their argument lies in the relation of $m_{E}$ with the Petersson  norm $||f_E||_{N}$, where $f_E\in S_2(\Gamma_0(N))$ be the newform associated to $E$. This association is due to Gross \cite{Gross85}, which we recall in Section~\ref{sec:tools}. When $E$ has a prime power conductor, Watkins's conjecture for any quadratic twist $E^{(D)}$ is proved by Caro in \cite{caro22watkinss}. Setzer's classification \cite{Setzer75} of such elliptic curves has proven useful, in addition to using the relation with $||f_E||_{N}$. In Section~\ref{sec:twisttrivial}, we extend the study of twists for the cases when $E(\Q)[2]$ is trivial. In particular, we record in Theorem~\ref{thm:twisttrivial} that Watkins's conjecture is true for almost all the twists under the assumption of GRH and BSD conjectures. 

Section~\ref{sec:tools} discusses several identities about $m_E$. Historically, they served as tools to study several arithmetic properties of $m_E$. As already mentioned, the relation between $m_E$ and $||f_E||_{N}$ is recalled, and the consequence of the growth of $m_E$ is discussed. Many resources are available on the related subjects, and nothing original is guaranteed in this section. Instead, it serves the reader a quick introduction to the subject and provides a notational setup for the article. In Section~\ref{sec:bounds}, we briefly discuss the growth of $m_E$.

In Section~\ref{sec:wff}, we undertake the study of the function field analog of Watkins's conjecture. Let $p>3$ be a prime and $k$ be a finite extension of $\mathbb{F}_p$. We define $A$ as the polynomial ring $k[T]$ and its field of fractions as $K = k(T)$. Let $E$ be an elliptic curve with conductor $\mathfrak{n}_E = \mathfrak{n}\infty$, where $\mathfrak{n}$ be an ideal in $A$, and $\infty$ denotes the place of $K$ associated to $T^{-1}$. We further assume that $E$ is non-isotrivial, i.e., the $j$-invariant is not in $k$ and has split multiplication reduction at $\infty$. Then we have an analog of the modular parametrization, i.e., a nonconstant morphism $\phi_E:X_0(\mathfrak{n})\to E$, where $X_0(\mathfrak{n})$ is the Drinfeld modular curve associated to $\mathfrak{n}$, see Section~\ref{sec:wff} for the definition. In this scenario, Caro~\cite{caro2022watkins} showed that an analog of Watkins's conjecture holds when we view $E$ over any extension $k'(T)$ of $K$, such that $k'$ contains all roots of the finite part of $\mathfrak{n}$. In this section, we point out an improvement over the imposed condition on $k'$ but at the cost of a few more additional conditions on $E$. We record this remark at Theorem~\ref{thm:watkinsf}. Furthermore, drawing the analogy with Section~\ref{sec:nontriv2t} and Section~\ref{sec:twisttrivial}, we discuss Watkins's conjecture for the families of twists $\{E^{(g)}\}$, for any polynomial $g$ having even degree. We recall Caro's results and mention some possible improvements. Furthermore, in section~\ref{sec:ffgrowth}, we discuss the growth properties of modular degree following \cite{PM2002}. It turns out that the modular degree is related to the special values of the Symmetric square $L$-function associated with $E$, whose analytic properties can be studied because Grothendieck's theory L-function allows one to write the $L$-function as a quotient of two polynomials. We shall note that the unpleasant factor in the lower bound, which is given by the non-separable degree of the morphism $j_E:\mathbb{P}^1 \to \mathbb{P}^1$, can be removed for a certain proportion of elliptic curves.

\subsection{Notations}
We write $f\sim g,$ if their domain is in $\mathbb{R},$ and they are asymptotically equivalent, that is, if $\lim\limits_{x\to \infty}\frac{f(x)}{g(x)}=1.$ We write $f \ll g$ for $|f|\leq c|g|$ where $c$ is a constant irrespective of the domains of $f$ and $g$, often $f=O(g)$ is written to denote the same. Moreover, we denote $f=o(g)$ when $\lim\limits_{x\to \infty}\frac{|f(x)|}{|g(x)|}= 0$. For any integer $n,$ we denote $\omega(n)$ as the number of distinct prime factors of $n$. We denote $\Q$ as the set of rational numbers and $h$ as the standard height on $\Q$. For any prime $p$, denote $\nu_p(\cdot)$ to be the associated $p$-adic valuation on $\mathbb{Q}$, and $\Q_p$ to be the field of $p$-adic numbers, and $\Z_p$ to be the ring of $p$-adic integers. For any variety $V$ over any field $K$, denote $V(K)$ be the set of all $K$-rational points on $V$.

\section{Tools to compute the modular degree}\label{sec:tools}
Let $E/\Q$ be an elliptic curve. This section briefly overviews the identities involving $m_E$, which we shall use throughout the rest of the article. 
\subsection{On the modular approach}
The Falting height $h(E)$ is defined to be 
$$h(E):=\frac{1}{12}\left(\log |\Delta_E|-\log |\Delta(\tau_E)\mathrm{im}(\tau_E)^6|\right)-\log (2\pi),$$
where $\tau_E\in \mathbb{H}$ corresponds to $E$, $\Delta(z)=q\prod_{n=1}(1-q^n)^{24}$ be the Ramanujan $\Delta$-function, and $j(z)=q^{-1}+\Z[[q]]$ be the normalized modular $j$-function. For semistable $E$, it turns out that $h(E)$ is close to $\frac{1}{12}h(j_{E})$. More precisely, $|h(E)-\frac{1}{12}h(j_{E})|\ll \log (h(j_E))$. In general, one extra factor comes from the \text{unstable} discriminant $\gamma$ of $E$. If we write $j_E=\frac{a}{b}$ in the minimal form, then $\gamma$ is denoted to be $\frac{\Delta_E}{b}$. In the semistable case, we have $\gamma=\pm 1$. In general, we have $|h(E)-\frac{1}{12}h(j_{E})-\frac{1}{12}\log \gamma|\ll \log (h(j_E))$. These estimates are useful when $E$ is in a nice enough Weierstrass form.

Let $\omega_E$ be the invariant differential associated with $E$. Given a modular parameterization $\phi_E: X_0(N)\to E$, one can write $\phi^{*}_E(\omega_E)= 2\pi ic_Ef_E(z)dz$, for some integer $c_E$. This constant (adjusting sign) $c_E>0$ is known to be Manin's constant. It is expected that $c_E= 1$ for the $\Gamma_0(N)$-optimal elliptic curve $E$. This is true when $E$ is semistable. In general, it is known \cite[Theorem 1.1]{pasten2018shimura} that given any finite set of primes $S$, we have $c_{E}=O_{S}(1)$, for any elliptic curve $E/\Q$ that is semistable outside $S$, i.e., $p^2\mid N_E \implies p\in S$.

Now let $f$ be any cuspform of level $N$. Then the Petersson  norm $||\cdot||$ is defined by
$||f||_N=(\int_{Y_0(N)}|f(z)| dx dy)^{1/2}$.
This can be computed explicitly by Rankin's method, which is given by
$$||f||_N=\frac{1}{48\pi}[\mathrm{PSL}_2(\Z):\Gamma_0(N)]~ \underset{s=2}{\mathrm{Res}}\left(\sum_{n=1}^{\infty}\frac{|a(n)|^2}{n^s}\right).$$
It is then due to Deligne \cite{Delinge85} that 
\begin{equation}\label{eqn:height} 
m_E=c^2_E4\pi^2||f_E||^2_Ne^{2h(E)}=\frac{4\pi^2 ||f_E||^2_N}{\mathrm{Vol}(E)},
\end{equation}
where we denote $\mathrm{Vol}(E)$ to be the area of the fundamental parallelogram associated with $E$.

The analytic symmetric-square $L$-function defined to be $L^A(Sym^2E,s)=\prod_{p} L_p^A(Sym^2E,s),$
where $L_p^A(E,s)=(1-\alpha^2_p/p^{s})^{-1}(\alpha_p\beta_p/p^s)^{-1}(1-\beta^2_p/p^s)^{-1}$. We have $\alpha_p = \overline{\beta_p}$ and $\alpha_p+\beta_p$ is the p th trace of Frobenius of $E$.
Additionally, if  $p^2\mid N_E$ then 
 $\alpha_p=\beta_p=0$ . For the description of $\alpha_p$ and $\beta_p$, please refer to \cite[Section 2]{MW02}.
A significant connection exists between symmetric squares and modular degree, established by Shimura \cite{shimura76}, given by:
\begin{equation}\label{analytic-l-function}
       \frac{L^{A}(Sym^2E,2)}{\pi i \mathrm{Vol}(E)}=\frac{m_{E}}{Nc_E^2}.
\end{equation}
unfortunately $L^A(Sym^2E,s)$ does not satisfy the function equation. For this to satisfy the functional equation, we adjust   $L^A(Sym^2E,s)$ with fudge factors, we get motivic L-function  $L^M(Sym^2E,s)=L^A(Sym^2E,s) U(s).$

It is worthwhile to note that the local  Euler factors $L_p^M(Sym^2E,s)=L_p^A(Sym^2E,s)$ whenever $p^2\nmid N_E$. If $p^2$ divides the conductor then $L^A_p(Sym^2E,s)=1$, which implies that $L^M_{P}(Sym^2E,s)=U_p(s)$.
From equation (\ref{analytic-l-function}), we can derive the following expression:
\begin{equation}\label{eqn:l-value}
\frac{L^M(\mathrm{Sym}^2E, 2)}{\pi i \mathrm{Vol}(E)}=\frac{m_E}{Nc_E^2}\prod\limits_{p^2\mid N}U_p(1).
\end{equation}

\subsection{Gross's association and supersingular zeroes}\label{sec:gross}
In this section, we assume that $N:=p$ is a prime. Denote $f_E$ as the associated newform of weight $2$ and level $p$. Gross in \cite{GK92} associated $f_E$ to an element $v_E$ of $\mathrm{Pic}(X)$, for some suitable curve $\mathrm{X}$. The curve $\mathrm{X}$ is $X_0(p)\pmod p$, which consists of two copies of $X_0(1)$ that meet transversally at the points whose underlying elliptic curve is supersingular at $p$. The Picard group $\mathfrak{X}$ of $\mathrm{X}$ can be written as $\sum_{i=1}^{n}\mathbb{Z}\cdot e_i$, 
where $n=g(X_0(p))+1$, and each $\{e_i\}$ can be identified with the isomorphism class of all supersingular elliptic curves over $\overline{\mathbb{F}_p}$. There is $\mathbb{Z}$-linear bilinear map
$$\langle -,-\rangle: \mathrm{Pic}(X_0(p))\times \mathrm{Pic}(X_0(p)) \to \mathbb{Z},$$
given by $\langle e_i, e_i \rangle =w_i\delta_{i,j}$, where $w_i=\frac{1}{2}\# \mathrm{Aut}(E_i)$. Now consider the Hecke correspondence $t_m:\mathfrak{X}\to \mathfrak{X}$ defined by,
$$t_m(e_i)=\sum_{j=1}^{n}B_{i,j}(m)e_j,$$
where $B_{i,j}(m)$ is the number of $m$-isogenies $E_i\to E_j$. The transposition of the matrix $(B_{i,j}(m))_{1\leq i,j\leq }$ is commonly known as the Brandt matrix.

Let $v_E=\sum v_E(e_i)e_i\in \mathfrak{X}$ be the vector associated to $f_E$, i.e., $t_m(v_E)=a(m)v_E$, where $a(m)$ is the $m^{\mathrm{th}}$ Fourier coefficient of $f_E$. We then have
$$m_E|E(\mathbb{Q})_{\mathrm{tor}}|=\langle v_E, v_E \rangle=\sum_{i=1}^{n} v_{E}(e_i)^2w_i,$$
where the first equality is due to Mestre~\cite{Mestre86}. 

There is an explicit description of $w_i$. If $p=1\pmod 4$, all of the $w_i$ are $1$. If $p=7 \pmod {12}$, all $w_i$ are $1$ except at one index $i$, in which case $j(e_i)=0$, and $w_i=3$. Otherwise if $p=11 \pmod {12}$, all $w_i$ are $1$ except at two indices $i,k$, in which case, $j$ and $w$ values are respectively, $0,1728$ and $3,2$. 

\subsection{Relations to congruence number $r_E$}
Doi, Hida, Ribet, Mazur, and others investigated congruence primes \cite{DO76, R83}; moreover, congruence primes were an essential component in Wiles's \cite{Wiles95} proof of  Fermat's last theorem. 
\begin{definition}
Let the newform linked to the elliptic curve $E$ be $f_E = \sum_{n>0}a_nq^n \in S_2\left(\Gamma_0(N) \right) $. 
The greatest integer such that there is a cuspform $g(z)=\sum_{n>0}b_nq^n$  with integer Fourier coefficients $b_n$, orthogonal to $f_E$ with respect to the Petersson  inner product, and congruent to $f_E$ modulo $r_E$, is the congruence number $r_E$ of the elliptic curve  $E$.
 (i.e., $a_n \equiv b_n \pmod{r_E}$ for all $n$).
\end{definition}
Finding the link between congruence numbers and the modular degree is quite intriguing. Ribet showed that $m_E\mid r_E$. It was later extended by Amod, Ribet, and Stein \cite{ARS} showing that if $\nu_p(N) \leq 1$, then $\nu_p(r_E) = \nu_p(m_E)$. Now one may naturally ask the following.
\begin{question}
        What are the prime factors of the quotient $\frac{r_E}{m_E}$?
\end{question} 
Amod, Ribet, and Stein \cite{ARS}  establishes that when the conductor of an elliptic curve is squarefree, indicating that the curve is semistable, the values of ${r_E}$ and ${m_E}$ are equal. Abbes and Ullmo, in their work \cite{AU96}, have proven that if a prime number $p$ does not divide the conductor $N$, then $\nu_p(m_E)=\nu_p(r_E)$. Amod, Ribet, and Stein \cite{ARS}, in their conjecture, proposed that $\nu_p\left(\frac{r_E}{m_E}\right)\leq \frac{1}{2}\nu_p(N)$. In other words, they suggest that the order of $p$ dividing the ratio $\frac{r_E}{m_E}$ is at most half the order of $p$ dividing the conductor $N$.
They also proved that   $p$ is a prime such that $p \mid r_E$ but $p \nmid m_E$ then $\dim_{\T/\mathfrak{m}} J_0(N)[\mathfrak{m}] > 2$, {\it i.e.},  multiplicity one fails for $\mathfrak{m}$. Where $\T$ is the Hecke algebra acting on $E$, and  $\mathfrak{m}$ be  the annihilator of $E[p]$   in $\T$. The ideal $\mathfrak{m}$  is said to  satisfies multiplicity
one  if $\dim_{\T/\mathfrak{m}} J_0(N)[\mathfrak{m}] =2$.

\subsection{ Congruence number of modular Abelian varieties}

We can define congruence numbers and modular degrees for modular Abelian varieties, which are quotient varieties of the Jacobian of modular curves.
For every $f\in S_2(\Gamma_0(N))$, define $I_f$ to be subset of  the hecke algebra  $T$ acting on $S_2(\Gamma_0(N)$ which annihilate $f$,
i.e $I_f=\{ T\in \T \big \vert  T(f)=0\}$. If $f \in S_2(\Gamma_0(N))$ is a modular form associated with modular abelian variety $A$, then $A\equiv J_0(N)/I_fJ_0(N)$.
Let $\phi_1: J_0(N) \to A$  be the quotient map, then  by Poincare reducibility over $\Q$, there exists a unique abelian subvariety $A^{\vee}$ of $J$ such that projects isogenously to the quotient $A$.
Let $\phi$ be the composite isogeny of 
\begin{align*}
    \phi:A^{\vee}\xrightarrow{\phi_1} J \xrightarrow{\phi_2} A.
\end{align*}
\begin{definition}
    The modular exponent of $A$ is the exponent of the kernel $\phi$. 
The {\it modular number } of $A$ equals the modular degree of $\phi$.
\end{definition}
When $A$ is an elliptic curve, the modular exponent is equal to the modular degree, and the modular number is the square of the modular degree. For more details about the modular degree of abelian varieties and  Watkins's conjecture for abelian varieties, refer to the article by Dummigan and Krishnamoorthy\cite{DK13}. Particularly \cite[Corollary 9.5]{DK13} discusses the $2$ power dividing the modular degree.

\begin{definition}\label{def:congabv}
    The {\it congruent number} of abelian variety $A=J_0(N)/I_fJ_0(N)$ is defined to be  order of the quotient group \begin{align}
    \frac{S_2(\Gamma_0(N),Z)}{S_2(\Gamma_0(N),Z)[I_{f}] \oplus S_2(\Gamma_0(N),Z)[I_{f}]^{\perp}}. 
\end{align}
The exponent of the above group is called the congruence exponent of the abelian variety $A$.
\end{definition}
The natural question is to ask how the congruence number of Abelian varieties is related to the congruence number of elliptic curves when the Abelian variety is an elliptic curve. To answer this question, we state the following theorem, stated in \cite{AU96} without proof. We provide a brief proof here. 

\begin{theorem}\label{congurence:number}
Let $f_E = \sum_{n>0} a_nq^n \in S_2(\Gamma_0(N ))$ be the newform linked to the elliptic curve $E$. Then the congruent number of the elliptic curve is consistent with Definition~\ref{def:congabv}.
\end{theorem}
\begin{proof}
Let $d$ be the order of the  quotient group $\frac{S_2(\Gamma_0(N),\Z)}{S_2(\Gamma_0(N),\Z)[I_{f_E}] \oplus S_2(\Gamma_0(N),\Z)[I_{f_E}]^{\perp}}$.
For any $g \in S_2(\Gamma_0(N),\Z)[I_{f_E}]^{\perp}$, we can find an $h \in S_2(\Gamma_0(N),\Z)$ such that $dh = -g + f_E$. This implies that $f_E\equiv g \pmod{d}$. Therefore, we have $d\leq r_E$.

Since $r_E$ is the congruence number, there exists $g\in S_2\left(\Gamma_0(N) \right)$ such that $f \equiv g \pmod{r_E}$. Now, define $h$ to be $\frac{1}{r_E} \left( f-g\right)$. We observe that $r_E$ is the smallest integer such that $r_E h \in S_2(\Gamma_0(N),\Z)[I_{f_E}] \oplus S_2(\Gamma_0(N),\Z)[I_{f_E}]^{\perp}$. Therefore, $d\leq r_E$, and the proof is complete.
\end{proof}
By the above theorem, the congruence number of Abelian varieties is equal to the congruence number of elliptic curves when the Abelian variety is an elliptic curve. 
Amod, Ribet, and Stein \cite{ARS} have also established analogous theorems to  $m_E\mid r_E$  for modular Abelian varieties. 
\subsection{A lower bound on $2$-valuation}\label{sec:2lbd}
Let $E$ be a semistable curve with the conductor $N$.
For each $d\mid \mid N$, let us consider the Atkin-Lehner involution $W_d$ on $X_0(N)$.  The set $\left\{W_d | d\mid \mid N\right\}$ forms an abelian  subgroup of automorphisms of rank equal to the number of prime dividing $N$. For every  prime divisor $p \mid N$, we have $W_p(f_E) = \pm f_E$. More precisely, there exists $w_p(f_E)\in \{\pm 1\}$ such that $W_p(f_E)=w_p(f_E) f_E$. It turns out that
$w_p(f_E)=1$, if $E$ has non-split multiplicative reduction at $p$, and $-1$, if $E$ has split multiplicative reduction. If $W_p(f_E) =  f_E$ Then the parameterization map $\phi: X_0(N)\to E$ factors through  $X_0(N)/W_p$.

\begin{center}
\begin{tikzcd}
X_0(N) \arrow[rd] \arrow[rr] &                       & E \\
                             & X_0(N)/W_p \arrow[ru] &  
\end{tikzcd}
\end{center}
 \noindent
The degree of the map $X_0(N) \to X_0(N)/W_p$ is $2$, contributing a factor of $2$ to the modular degree.
Let $\mu$ be the  number of primes dividing $N$.
Let us consider the homomorphism $(\Z/2\Z)^{\mu} \to \{\pm 1\}$ defined by, $W_d \mapsto w_d$. Suppose $W'$ is the kernel of this map. Since the field generated by the Fourier coefficients of the associated newform $f_E$ is $\Q$, it follows from Dummigan-Krishnamoorthy \cite[Proposition 2.1]{DK13} that there is a homomorphism $W'\to E(\Q)[2]$ with kernel $W''$, whose order divides the modular degree $m_E$. The order of $W''$ is at least greater than or equal to $\frac{\#W'}{E(\Q)[2]}$.  This gives,
\begin{align*}
    \nu_2(m_E)&\geq \nu_2\left(\frac{\#W'}{\#E(\Q)[2] }\right)\\
    &  =\nu_2\left({\#W'}\right) -\nu_2\left( \#E(\Q)[2] \right)\\
&=\mu- \dim_{\Z/2\Z}(W/W')-\mathrm{dim}_{\mathbb{Z}/2\mathbb{Z}}(E(\mathbb{Q})[2]).
\end{align*} 
\section{Divisibility and Watkins's conjecture}
In this section, we shall discuss the viability of Watkins's conjecture. Recall from the introduction that Watkins's conjecture predicts that $2^{\mathrm{rank}(E(\mathbb{Q}))}$ should divide $m_E$. We shall briefly discuss the approaches that have been taken to study this.
\subsection{Heuristic proof of Watkins's conjecture}
Dummigan heuristically proved that certain classes of elliptic curves satisfy Watkins's conjecture. Let $E$ be an elliptic curve with the squarefree even conductor $N$, with no rational point of order 2, and $E[2]$ is ramified at
all primes $p\mid N$ and the action of complex conjugation on $E[2]$ is non-trivial, then Dummigan~\cite{neil06} proved that $E$ satisfies Watkins's conjecture assuming certain $R\cong T$.
He used Galois cohomology techniques and a squaring map to capture the weak Mordell-Weil group in a Selmer group. Later he established the isomorphism with the tangent space of the universal deformation ring of residual representation $\overline{\rho}: \Gal(\overline{\mathbb{Q}}/{\mathbb{Q}}) \to \mathrm{GL}_2(E[2])$ with certain deformation conditions. Local conditions in the $2$-Selmer group determine the deformation conditions. Later, using Wiles's numerical criterion, he proved that  $2^{\rank\left(E(\Q)\right)}$ divides the modular degree of the elliptic curve.

In support of the conjecture, let us note that the conjecture is true for almost all elliptic curves over $\Q$ when arranged by a suitable height. Any elliptic curve over $\Q$ can be uniquely written in the Weierstrass form 
\begin{equation}\label{eqn:minform}
E_{A, B}:y^2=x^3+Ax+B,~A, B\in \Z~\mathrm{such~that}p^6~\mathrm{does~not~divide}~B~\mathrm{whenever}~p^4\mid A.
\end{equation}
When we organize elliptic curves over $\Q$ as in $(\ref{eqn:minform})$, using the height defined as $H(E_{A,B})=\max\{|A|^3,|B|^2\}^{1/6}$, the quantity of such elliptic curves that have a height not exceeding $x$ is asymptotic to $cx^{5}$ for some positive constant $c$. 

Denote,
$$\mathrm{Avg}(\mathrm{rank})=\mathrm{lim}~\mathrm{sup}_{x\to \infty}~\frac{\sum_{H(E)\leq x} \mathrm{rank}(E)}{\sum_{H(E)\leq x} 1},~\mathrm{Avg}_2(\mathrm{degree})=\mathrm{lim}~\mathrm{sup}_{x\to \infty}~ \frac{\sum_{H(E)\leq x} \nu_2(m_E)}{\sum_{H(E)\leq x} 1}.$$
Note that Watkins's conjecture implies that $\mathrm{Avg}(\mathrm{rank})\leq \mathrm{Avg}_2(\mathrm{degree})$. However, we can still prove the consequence without using Watkins's conjecture. In some sense, the following can be considered an average version of Watkins's conjecture.
\begin{theorem}\label{thm:avg}
Over the family of all elliptic curves/~$\mathbb{Q}$, we have the following
    $$\mathrm{Avg}(\mathrm{rank})\leq \mathrm{Avg}_2(\mathrm{degree}).$$
\end{theorem}
To prove this, we first need the following key ingredients. 
\begin{lemma}\label{lem:highomega}
    When arranged by height, for almost all the elliptic curves $E/\Q$ we have $\omega(N_E)\geq 2$.
\end{lemma}
However, we include a self-contained proof due to needing a reference.
\begin{proof}
Consider an elliptic curve $E_{A,B}$ as in $(\ref{eqn:minform})$. Then the conductor $E_{A,B}$ is (up-to some bounded factor of $2$ and $3$) the product over all primes
$p$ dividing the discriminant $\Delta(E_{A,B})=-4A^3-27B^2$. Let us denote
$$\mathcal{E}(x)=\{A,B\in \Z: E_{A,B}~\mathrm{is~as~in}~(\ref{eqn:minform}),~H(E_{A,B})\leq x\}.$$
Then, it is enough to show that
\begin{equation}\label{eqn:toshow}
N(x)=\#\{A,B\in \mathcal{E}(x):\omega\left(\frac{4A^3+27B^2}{2^{\nu_2(4A^3+27B^2)}3^{\nu_3(4A^3+27B^2)}}\right)\leq 1\}=o(x^5).
\end{equation}
Consider a parameter $z>3$, which will be specified later. If a prime $p>z$ divides $4A^3+27B^2$ for some integers $A,B$, then we have $\left(\frac{-3A}{p}\right)=1$. For any integer $A$, consider $\mathcal{P}_A$ to be the set of primes $p$ for which $\left(\frac{-3A}{p}\right)=1$. We have the following estimate due to Selberg's sieve,
$$S(A,\mathcal{P}_{A},z)=\#\{|B|\leq x^3: p\mid 4B^2+27A^3,~p>3\implies p>z\}= \frac{x^{3}}{
\log z
}+O(z^2).$$
In particular, 
\begin{equation}\label{eqn:sieve}
   \#\{A,B\in \Z: |A|\leq x^2,~|B|\leq x^3: p\mid 4B^2+27A^3,~p>3\implies p>z\}=\frac{x^5}{\log z}+O(x^2z^2). 
\end{equation}
Now, using (\ref{eqn:sieve}), the quantity $N(x)$ at (\ref{eqn:toshow}) can be estimated as follows
\begin{equation}\label{eqn:for1}
N(x)\leq \sum_{p<z} \#\{A,B\in \Z : E_{A,B}\in \mathcal{E}(x),~\frac{4A^3+27B^2}{2^{\nu_2(4A^3+27B^2)}3^{\nu_3(4A^3+27B^2)}}=\mathrm{a~power~of}~p\}+\frac{x^{5}}{
\log z
}+O(x^2z^2).
\end{equation}
Note that for any integers $n,e,f$, and any prime $p>3$, we have $|4A^3+27B^2|=2^e3^fp^{n}$ for at most $x^{3}$ many integers $A,B\in \Z$ with $|A|\leq x^2, |B|\leq x^3$. There are at most $O(\log x)$ many choices for each of the $e,f$, and $n$, and at most $z$ choices for $p$. This shows that, $N(x)\ll z(\log x)^3+\frac{x^{5}}{
\log z
}+x^2z^2$. In particular, the result follows taking $z=x^{3/2-\varepsilon}$.
 \end{proof}
With this, we are now ready to prove Theorem~\ref{thm:avg}
\begin{proof}[Proof of Theorem~\ref{thm:avg}]
    It follows from Lemma~\ref{lem:highomega} that the conductor of almost the elliptic curves has at least two prime factors. Moreover, it follows from Grant's result in \cite{grant00} that almost all of them are $2$-torsion free over $\Q$. Now, the discussion in Section~\ref{sec:2lbd} implies that $\mathrm{Avg}_2(\mathrm{degree})\geq 1$. The proof from Bhargava, Shankar's result in \cite[Theorem 3]{BS13} is complete.
\end{proof}
\begin{remark}\rm
    Minimalist conjecture predicts that the average rank of elliptic curves should be 1/2; in particular, half of the curves have a rank $0$, and half have rank $1$. Since Watkins's conjecture is known to be true for all curves of rank $1$, we see that Watkins's conjecture is true for almost all elliptic curves. This is a stronger version of Theorem~\ref{thm:avg}.
\end{remark}

\subsection{Elliptic curves with odd modular degree}\label{sec:odddeg}
Calegari and Emerton undertook a significant endeavour to establish the validity of Watkins's conjecture, focusing specifically on elliptic curves with odd modular degrees. Their approach involved developing a distinct characterization for elliptic curves falling within this category. Through an in-depth analysis of the Atkin–Lehner involution on $X_0(N)$, they effectively demonstrated that when an elliptic curve $E$ has an odd modular degree, the curve's conductor $N$ is divisible by a maximum of two odd primes. Additionally, they established that the analytic rank of $E$ is necessarily even. They also show that if $E$ is an elliptic curve with an odd modular degree whose conductor $N$ has more than one prime factor, then $E$ possesses a rational point of order 2. 

Let $\T$ denote the Hecke algebra over $\Z$ acting
on  $S_2\left(\Gamma_0(N)\right)$, and let $\mathfrak{m}$ be a maximal ideal of $\T$ such that $\T/\mathfrak{m} = \mathbb{F}_2$.
 Let $f_E$ be the eigenform of level $\Gamma_0(N)$ and weight $2$ corresponding to the elliptic curve $E$. According to the theorem in \cite{ARS}, the modular degree $m_E$ of $E$ is divisible by $2$ if and only if $f_E$ satisfies a congruence modulo $2$ with a cusp form of level $\Gamma_0(N)$. The set of cuspidal eigenforms congruent to $f_E$ is indexed by the set $\mathrm{Hom}(\mathbb{T}_{\mathfrak{m}}\times \mathbb{Q}_2,\mathbb{Q}_{2})$. If $f_E$ does not satisfy any congruence with a cusp form, it means that there are no cuspidal eigenforms congruent to $f_E$. In other words, the set $\mathrm{Hom}(\mathbb{T}_{\mathfrak{m}}\times \mathbb{Q}_2,\mathbb{Q}_2)$ is trivial.
Therefore, $f_E$ satisfies no congruence with a cusp form of level $\Gamma_0(N)$ if and only if $\mathbb{T}_\mathfrak{m} = \mathbb{Z}_2$.
It is known that $\mathbb{T}_\mathfrak{m} = \mathbb{Z}_2$ equivalent to  that there exists no  nontrivial
minimal deformation of  $\GL_2(\mathbb{F}_2[x]/(x^2))$ of $\overline{\rho}$ , where $\overline{\rho}: \Gal(\overline{\mathbb{Q}}/{\mathbb{Q}}) \to \mathrm{GL}_2(E[2])$ is the residual representation corresponding to $E[2]$.

The case of an elliptic curve $E$ with an odd modular degree and a conductor that is a prime has been analyzed using the Galois deformation theory. Furthermore, they demonstrated that if the Galois representation $\overline{\rho}$ satisfies any of the following conditions:
\begin{itemize}
    \item $\overline {\rho}$ is totally real.
    \item $\overline {\rho}$ is unramified at 2.
  \item  $\overline {\rho}$ is ordinary, complex, and ramified at 2.
\end{itemize}

then the ring of Hecke operators $\mathbb{T}_\mathfrak{m}$ is not isomorphic to $\mathbb{Z}_2$. Consequently, it can be concluded that if $E$ has a prime conductor, it exhibits supersingular reduction at 2, and $\Q(E[2])$ is entirely complex.

In the analysis of the case where the conductor is a prime power, it was observed that there exist only finitely many elliptic curves of conductor $2^k$ for all positive integers $k$. Now, let's consider the case where the conductor is of the form $N=p^r (r>1)$, with $p$ being an odd prime. By considering the twist of $E$ by $\chi$, which is the unique quadratic character of conductor $p$, and noting that quadratic twists preserve $E[2]$,  we get   $f_E \equiv f_E\otimes \chi \pmod{2}$. Since $E$ has an odd modular degree, this implies that $f_E = f_E\otimes \chi$. Therefore, $E$ has complex multiplication. It is known that there are only finitely many elliptic curves with complex multiplication and conductor $p^r$. The authors determined the modular degree of these curves by consulting current databases and found that the only possible values for $N$ are 27, 32, 49, or 243.

Subsequently, building upon the work of Calegari and Emerton \cite{CE09} and Agashe et al. \cite{ARS}, Yazdani examined the elliptic curves with odd congruence numbers. They found that the conductors of all such elliptic curves, with a few exceptions, can be categorized as $p, pq, 2p, 4p$.  Each of these categories was meticulously examined to demonstrate that all elliptic curves falling within these classifications possess a finite Mordell-Weil group, with the only potential exception being when the conductor is prime. 
In other words, Yazdani also proved that if an elliptic curve $E$ defined over $\mathbb{Q}$ has an odd modular degree, then its rank is zero unless it possesses a prime conductor and an even analytic rank. For a more comprehensive statement, please refer to \cite[Theorem 3.8]{yazdani11}.

Furthermore, Kazalicki and Kohen \cite{KK18watkinscorrigendum} settled the case of elliptic curves with prime conductors by proving that those with an odd congruent number have a rank of zero. 
To establish their results, Kazalicki and Kohen utilized the Gross association, as described in Subsection \ref{sec:gross}, to establish a connection between the elliptic curve and an element of the Picard group. They then proved that if a rational elliptic curve $E$ has a prime conductor and a non-zero rank, the modular degree of $E$ is always even. Let us give a brief idea of the argument in the next section.

\subsection{Elliptic curves with prime conductor}\label{sec:primecond}
Let $E/\Q$ be any elliptic curve of positive rank, Mestre's result in \cite{Mestre86} implies that 
$m_E=\langle v_E, v_E \rangle=\sum_{i=1}^{n} w_iv_{E}(e_i)^2$. On the other hand, it turns out that
$\sum_{i=1}^{n}v_E(e_i)=0$, and all the $w_i$ are $1$ when $p=1 \pmod 4$. We immediately have that $m_E$ is even in this case. For the case $p=3 \pmod 4$, one needs to use \cite[Proposition 2.4]{KK18watkins} and Gross-Kudla's formula \cite[Proposition 2.5]{KK18watkins}. This settles down the conjecture for $\mathrm{rank}(E(\Q))=1$ case. For rank $2$ case, Gross-Kudla's formula implies $\sum_{i=1}^{n}v_E(e_i)^3=0$. The reader may take a look at \cite{KK18watkins} for more details. Consequently we have $\sum_{i=1}^{n}v_E(e_i)^2=0\pmod 4$, if each of the odd terms in $\{v_E(e_i)\}_{1\leq i
\leq n}$ appears an even number of times. This is literally true, provided that the following conjecture holds.
\begin{Conjecture}\label{conj:supsing}
Let $E/\Q$ be any elliptic curve of conductor $p$ and $\rank\left(E\left( \Q\right)\right) >0$. Then the coefficient $\nu_E(e_i)$ of $e_i$ in $v_{E}$ is even whenever $j(e_i)\in \mathbb{F}_p$.
\end{Conjecture}

When the root number of $E$ is $-1$, Kazalicki and Kohen showed that all such $\nu(e_i)$ are, in fact, $0$. This is true because the root number condition implies $t_p(v_E)=-v_E$. For each $1\leq i\leq n$, set $1\leq \bar{i}\leq n$ be the unique index such that $E_{\bar{i}}=E^p_i$. Note that, $i=\bar{i}$ if and only if, $j(E_i)\in \mathbb{F}_p$. With this, we have \cite{KK17singular} that $t_p(e_i)=e_{\bar{i}}$, which gives a proof of the conjecture for root number $-1$ case. 
In the root number $1$ case, the conjecture is known to be true if  $E$ has a positive discriminant and there  exists a prime $\ell$ such that
\begin{equation}\label{eqn:oddcond}
\left(\frac{-p}{\ell}\right)=-1,~a(\ell)=1 \pmod 2.
\end{equation}
Under these conditions, we have $B_{i,j}(\ell)=0\pmod 2$, whenever $i,j \in S_p$. Moreover, the conditions in (\ref{eqn:oddcond}) are satisfied for any $E/\Q$ of conductor $p$, satisfying $E(\Q)[2]$ is trivial and $\Delta_E>0$. 

The main ingredients are $B_{i,j}(\ell)=0\pmod 2$ and 
$a(\ell)=1 \pmod 2$.
In the case of $\Delta_E<0$, the condition $\left(\frac{-p}{\ell}\right)=-1$ implies $\ell$ is inert in $\Q(\sqrt{-p})$ which implies that $2$ divides the order of  $\mathrm{Frob}_{\ell}$, which implies that $a(\ell)$ is not odd. This is one of the main hurdles in the case of root number $1$ and negative discriminant. However, Kazalicki and Kohen's proof also shows that it is enough to have some integer $m$ for which $a(m)=1 \pmod 2$ and all the $B_{i,j}(m)$ are even. From the commutativity of Brandt's matrices, and \cite[Proposition 20]{KK17singular} it follows that $\widetilde{B}(m)=\left(B_{i,j}(n)\right)_{i,j\in S_p}\pmod 2$ is also commutative inside $M_{s_p}(\Z/2\Z)$. Therefore it may be fruitful to analyze the structure of those $\widetilde{B}(m)$, whose one of the eigenvalues is $1$. Also, note that each $B_{i,j}(m)$ gives a modular form $\theta_{i,j}$, from which we get another modular form $\theta'_{i,j}:=\theta_{i,j}-E_2$, where $E_2=1-24\sum_{n=1}^{\infty}\frac{nq^n}{1-q^n}$. Now we require the cusp forms $f_E$ and $\theta'_{i,j}$ to have a common Fourier coefficient of opposite parity. A statement about the independency of the Fourier coefficients, e.g., \textit{Generalized Sato-Tate hypothesis} \cite[Section 4.1]{BBG22}, could be helpful in this approach.      

We end the discussion of this section with the following remark. As long as the prime conductors are concerned, this is predicted by Watkins in \cite[Section 4.2]{MW02} the following.
\begin{Conjecture}
    Let $N:=p$ be a prime, and $\nu_2(m_E)=0$, then either $p=17$, or $p\equiv 3 \pmod 8$, or if $p$ is of the form $x^2+64$.
\end{Conjecture}
\subsubsection{A conjecture on the supersingular coefficients}
In this section, we discuss the Conjecture~\ref{conj:supsing}. In \cite{KK17singular}, Kazalicki and Kohen studied modular forms over $\overline{\mathbb{F}_{p}}$ in the sense of Katz's modular form. If we are a bit more specific, there exists a polynomial $P_E(x)\in \Z[x]$ whose roots in $\mathbb{F}_p$ are often supersingular zeros. In particular, they showed that if $j\in \mathbb{F}_p$ is $j$ invariant of a supersingular elliptic curve $E_i$ over $\overline{\mathbb{F}}_p$, then $P_E(j)=0\pmod p$ if and only if, $v_{E}(e_i)=0 \pmod p$. In the sense of Katz, the condition $P_E(j)=0\pmod p$ corresponds to a rule that is $0$ acting on $E_i$. Finally, a study of Hecke operators on the space of modular forms corresponding to supersingular elliptic curves, again in the sense of Katz. This, of course, does not settle the conjecture; however, this shows that in the case of root number $-1$ case, any supersingular $j$ invariant in $\mathbb{F}_p$ is a root of $P_E(x)\pmod p$. We denote $s_p$ to be $\# S_p$, i.e., the number of isomorphism classes of supersingular elliptic curves over $\mathbb{F}_p$. Following a remark in Kazalicki-Kohen in \cite{KK17singular}, we can deduce the following in support of the conjecture.
\begin{theorem}\label{thm:supsingconj}
   Let $E/\Q$ be any elliptic curve of positive rank having prime conductor $p$. Take any prime $p$ satisfying $\left(\frac{-D}{p}\right)= -1$, for some $D\in \{1, 3, 7, 11, 19, 43, 67, 163\}$. Then the following implications hold.
   \begin{enumerate}
\item[(a)] Conjecture~\ref{conj:supsing} is true for any such prime $p$ with $s_p=2$.
       \item[(b)] If $s_p>2$, and  $s_p$ is even, there are at least two distinct $i,j\in S_p$ for which $v_E(e_i)=v_E(e_j)=0\pmod 2$.
   \end{enumerate}

\end{theorem}
Before proving this, let us first define the required terminology and recall the preliminary results related to the supersingular elliptic curves. For any integer $D=3 \pmod 4$, we say that an elliptic curve over any field $F$ has CM over $\mathcal{O}_{-D}=\mathbb{Z}[\frac{1}{2}(D+\sqrt{-D})]$ if, $\mathcal{O}_{-D}$ contains maximally in the endomorphism ring of the elliptic curve over $\overline{F}$. Then supersingular elliptic curves with $j$-invariant in $\mathbb{F}_p$ correspond to those $D$ with those $D$ for which $p$ is ramified or inert in $\Q(\sqrt{-D})$. On the other hand, there are $h(-D)$ many $\overline{\Q}$ isomorphism classes of elliptic curves over $\overline{\Q}$ with CM by $\mathcal{O}_{-D}$, where $h(-D)$ is the class number of $\Q(\sqrt{-D})$. Furthermore, we denote $h_i(-D)$ to be the number of optimal embeddings of the order of discriminant $-D$ into $\mathrm{End}(E_i)$ modulo conjugation by $\mathrm{End}(E_i)^{*}$ and $u(-D)$ is the number of units of the order. Then we set $b_D=\frac{1}{u(-D)}\sum_{i=1}^{n}h_i(-D)v_E(e_i).$ 
\begin{proof}[Proof of Theorem~\ref{thm:supsingconj}]
    We first claim that the number of odd terms is always even. This is immediate from the fact that $\sum_{i=1}^{n} v_E(e_i)=0$ and that, $v_E(e_i)=v_E(e_{\bar{i}})\pmod 2$, in particular,
    $$\sum_{i\in S_p} v_E(e_i)\equiv 0 \pmod 2.$$
This shows that if there is one even $v_E(e_i)$, then there are at least two even terms since $s_p$ is even. Since $\mathrm{rank}(E(\Q))>0$, we have $L(E,1)=0$. In particular, by Gross-Waldspurger formula \cite{Gross85}, we have $\langle v_E,b_D\rangle =0$ for any $D>0$ with $\left(\frac{-D}{p}\right)=-1$. 

For proof of part (a), note that since $p$ and $D$ satisfy the imposed conditions, $p$ is inert in $\Q(\sqrt{-D})$. In particular, $b_D$ is supported only on the indices in $S_p$. Moreover, since $\Q(\sqrt{-D})$ has class number $1$, we have $h_i(D)=1$ for exactly one $i\in S_p$, and that $E_i$ has CM by $\mathcal{O}_{-D}$. Then the condition $\langle v_E,b_D \rangle=0$ implies $v_E(e_i)=0$. In particular, if $s_p=2$, then $v_E(e_{i'})=0$ for the other $i'\in S_p$, as well. Now for part (b), the proof follows from the observation in the previous paragraph since $s_p$ is even.

\end{proof}
\begin{remark}\rm
It is known \cite{Gross85} that 
\begin{align*}
    s_p 
    =\begin{cases}
			\frac{1}{2}h(-p), & \text{if $p\equiv 1 \pmod 4$}\\
           2h(-p), & \text{if $p \equiv 3 \pmod 8$}\\
           h(-p),& \text{if $p\equiv 7 \pmod 8$}\\
   		 \end{cases}
\end{align*}
\end{remark}
We turn into discussing the case of any arbitrary prime power. In this case, it turns out that $\mathrm{rank}(E(\Q))\leq 1$ provided that $E(\Q)[2]$ is non-trivial, see \cite[Lemma 2.4]{Setzer75} for details. In particular, the Watkins conjecture is true in this case. Furthermore, a complete classification of all such elliptic curves could be found in \cite{Setzer75}. The next section will briefly discuss the conjecture for any elliptic curves with a non-trivial $2$-torsion defined over $\Q$.

\subsection{Elliptic curves with non-trivial $2$-torsion}\label{sec:nontriv2t}
In this section, we assume that $E(\mathbb{Q})[2]$ is either $(\Z/2\Z)$ or $(\Z/2\Z)^2$. One of the important advantages of this assumption is that $2\mid \# E(\mathbb{F}_p)$ for any prime $p$ co-prime to $2N$, since $E(\mathbb{Q})[2]$ embeds in $E(\mathbb{F}_p)$. Also, in this case, we have a rank bound in terms of the number of prime factors of $N$. For any squarefree integer $D$, denote $E^{(D)}$ be the twist of $E$ by $D$. Recall that the quadratic twist $E^{(D)}$ is the elliptic curve defined by $y^2 = x^3 + aD^2x + bD^3$. Note that $E^{(D)}$ is isomorphic to $E$ over the field $\Q(\sqrt{
D})$. From (\ref{eqn:height}), we have 
\begin{equation}\label{eqn:lbdtwist}
\nu_2(m_{E^{(D)}})\gg \sum_{\substack{p\mid D\\ (p,~2N)=1}}\nu_2\left((p+1)(p+1-a(p))(p+1+a(p))\right)\gg 3(\omega(D)-\omega(N)).
\end{equation}
Then for any $D$ with a sufficiently large number of prime factors, the proof of Watkins's conjecture for $E^{(D)}$ (see \cite{EP21}) follows from the following rank bound
\begin{equation}\label{eqn:rankbdd}
\mathrm{rank}(E(\Q))\leq 2\alpha+\mu-1\leq 2\omega(N)-1,
\end{equation}
where $\alpha$ and $\mu$ respectively denote the number of additive and multiplicative reductions of $E$. In particular, given any $E/\Q$, the number of $|D|\leq x$ for which the twisted curve $E^{(D)}$ does not satisfy Watkins's conjecture is only of order $\frac{x(\log \log x)^{k(E)}}{\log x}$, where $k(E)=6+5\omega(N)-\nu_2(m_E/c^2_E)$. In the semistable case, or at least when $E$ is semistable away from a finite set, $k(E)$ is linearly bounded as a function of $\omega(N_E)$.

Recall that the rank bound at (\ref{eqn:rankbdd}) gives the proof of Watkins's conjecture when $\omega(N)=1$, i.e., when $N$ is a prime power. In fact, in the case when $E$ has a prime power conductor, it is shown in \cite{caro22watkinss} that any quadratic twist $E^{(D)}$ satisfies Watkins's conjecture. The proof essentially follows from the classification in \cite{Setzer75}, where a much stronger estimate holds in (\ref{eqn:lbdtwist}) and (\ref{eqn:rankbdd}).
More precisely, we have
$$\nu_2(m_{E^{(D)}})\geq \nu_2\left(\frac{m_E}{c^2_E}\right)+\nu_2\left(\frac{||f||^2_{N^{(D)}}}{||f_E||^2_N}\right)+\frac{1}{6}\nu_2\left(\frac{\Delta_{E^{(D)}}}{\Delta_E}\right).$$
Essentially from Setzer's classification, the following estimates \cite{caro22watkinss} hold
$$ \nu_2\left(\frac{m_E}{c^2_E}\right)\geq -1,~\nu_2\left(\frac{||f||^2_{N^{(D)}}}{||f_E||^2_N}\right)\gg \omega(D),~\nu_2\left(\frac{\Delta_{E^{(D)}}}{\Delta_E}\right)=O(1),$$
where the implicit constants are absolute.

In this direction, we also recall the result of Caro and Pasten in \cite{CP22}. They showed that $E$ satisfies Watkins's conjecture if $E(\Q)[2]=\Z/2\Z$, and all primes of bad reductions, are either non-split or if the number of primes of non-split reduction is odd. In this section, we shall give an extension of this result to the case $E(\Q)[2]=(\Z/2\Z)^2$.
\begin{theorem}\label{thm:CP22extension}
Let $E/\Q$ be any semistable elliptic curve such that $E(\Q)[2]=(\Z/2Z)^2$. Then Watkins's conjecture is true for $E$ provided that one of the following holds.
\begin{enumerate}
    \item[(a)] $E$ has non-split multiplicative reductions at all primes dividing $N_E$, and $\omega(N_E)$ is odd.
    \item[(b)] $E$ has at least one split multiplicative reductions, and $\mathrm{dim}_{\mathbb{F}_2}(\Sha[2])> 1$.
\item[(c)] Parity conjecture is true for $E,~\mathrm{dim}_{\mathbb{F}_2}(\Sha[2])= 1$ and number of non-split multiplicative reductions is even.
   \item[(d)]  $\mathrm{dim}_{\mathbb{F}_2}(\Sha[2])> 2$.
\end{enumerate}
Moreover, if $E(\Q)[2]=\Z/2\Z$, then Watkins's conjecture is true for $E$ provided that, either the conditions at \cite[Theorem 1.2]{CP22}, or if $\Sha[2]$ is non-trivial.
\end{theorem}
\begin{proof}
For the proof of part (a), we follow the same approach as in \cite{CP22}. If all $E$ has a non-split reduction at every prime, then 
$\nu_2(m_E)\geq \mu-2$. Suppose that $\mathrm{rank}(E(\Q))=\mu-1$, then it follows from the exact sequence
   \begin{align}\label{selmer:shortexactsequence}
       0\to E(\Q)/2E(\Q)\to \mathrm{Sel}^2(E)\to \Sha[2]\to 0,
   \end{align}
that $\dim_{\mathbb{F}_2}(\mathrm{Sel}^2(E))\geq \mu+1$. On the other hand, it immediately follows from the second exact sequence in \cite[Lemma 3.1]{CP22} that 
$$ \mathrm{dim}_{\mathbb{F}_2}(\mathrm{Sel}^2(E))\leq s(E,\theta)+s'(E,\theta)\leq  \mu+1.$$
This shows that $\Sha(E)[2]$ is trivial, hence, the parity conjecture holds for $E$, consequently $\mathrm{rank}$ is even, since $-1=-\prod_{p\mid N}w_p=w(E/\Q)=(-1)^{\mathrm{rank}(E(\Q))}$. This is a contradiction since the rank is $\mu-1$, which is even. Here $w(E/\Q)$ denotes the root number of $E/\Q$, and $w_p$ denotes the local root number, i.e., the sign of the involution of $W_p$ acting at $f_E$. 

For the proof of part $(b)$, we have $\nu_2(m_E)\geq \mu-3$. Let us assume that $\mathrm{rank}(E(\Q))\geq \mu-2$. From the proof of part $(a)$, we know that  $\mathrm{dim}_{\mathbb{F}_2}(\mathrm{Sel}^2(E))\leq \mu+1$. This shows that $\mathrm{dim}_{\mathbb{F}_2}(\Sha[2])\leq 1$, a contradiction, as desired.

Now for a proof of part $(c)$, note that the imposed condition on $\Sha$ and the fact that $\mathrm{dim}_{\mathbb{F}_2}(\mathrm{Sel}^2(E))\leq \mu+1$ implies $\mathrm{rank}(E(\Q))\leq \mu-2$. Since the parity Conjecture is true for $E$, and the number of non-split multiplication is even, we have $w(E/\Q)=(-1)^{\mu-n_s+1}=(-1)^{\mathrm{rank}(E(\Q))}$, and hence $\mu-\mathrm{rank}(E(\Q))$ is odd. Here $n_s$ denotes the number of non-split multiplicative reductions. This shows that $\mathrm{rank}(E(\Q))\leq \mu-3$, as desired.

To prove part $(d)$, observe that we have $\nu_2(m_E) \geq \mu - 3$ and $\text{rank}(E(\mathbb{Q})) \leq \mu - 1$. Now, assume for contradiction that $\text{rank}(E(\mathbb{Q})) > \mu - 3$. According to the short exact sequence (\ref{selmer:shortexactsequence}), it follows that $\dim(\mathrm{Sel}^2(E)) > \mu + 1$, which leads to a contradiction. Therefore, the assumption is false, and we can conclude that $\text{rank}(E(\mathbb{Q})) \leq \mu - 3$.

Now for the case $E(\Q)[2]=\Z/2Z$, it is enough to consider the case when there is at least one split multiplicative reduction. In this case, we have $\nu_2(m_E)\geq \mu-2$. If $\mathrm{rank}(E(\Q))\geq \mu-1$, then we must have $\mathrm{dim}_{\mathbb{F}_2}(\Sha[2])= 0$. This is because, $\mathrm{dim}_{\mathbb{F}_2}(\mathrm{Sel}^2(E))\leq \mu$, which follows from \cite[Lemma 3.1]{CP22}. Hence, a contradiction, as $\Sha[2]$ is non-trivial.
\end{proof} 
\begin{remark}\rm
\begin{enumerate}
\item[(1)] One can follow a similar treatment for the elliptic curves having additive reductions at some primes. In that case, one needs to put suitable conditions on the parity of $N^{+}-N^{-}$, where $N^{+}$ (resp. $N^{-1}$) denotes the number of $p^e$ such that $p^e\mid\mid N_E$ and $w_p=1$ (resp. $w_p=-1$). We refer the reader to \cite[Theorem 1.1]{rootnumber2000} for a description of the local root numbers at the primes of additive reductions.
\item[(2)] In examining the elliptic curve $E: y^2 + xy + y = x^3 + x^2 - 71x - 196$, reference to the LMFDB\footnote{Refer to \href{https://www.lmfdb.org/EllipticCurve/Q/3315/a/2}{https://www.lmfdb.org/EllipticCurve/Q/3315/a/2} for more details.} indicates that this curve is semistable and exhibits at least one instance of split multiplicative reduction, alongside possessing a rank of $2$. Further analysis conducted using SageMath reveals that the 2-Selmer rank associated with $E$ is $4$. This set of characteristics ensures that $E$ adheres to the prerequisites outlined in part (b) of Theorem~\ref{thm:CP22extension}.
\item[(3)] For any elliptic curve $E$ mentioned in Theorem~\ref{thm:CP22extension}, one can calculate the $2$-Selmer rank following the methodology outlined in \cite[Section 3]{RPC}. This calculation provides an estimate on the rank of the Shafarevich-Tate group $\Sha[2]$. However, this information is not entirely adequate unless we know the rank of the elliptic curve $E$ over the rational numbers $\Q$.

\end{enumerate}
\end{remark}
\subsection{Elliptic curves with trivial $2$-torsion}\label{sec:twisttrivial}
It is known  from \cite{grant00} that $E(\Q)[2]$ is trivial for almost all the elliptic curves over $\Q$. Therefore, we are considering the most elliptic curves in this case. In this cases, we have $\mathrm{Gal}(\Q(E[2])/\Q)$ is either $\Z/3\Z$, or symmetric group $S_3$. Note that $a(p)$ is even, if and only if, $\mathrm{Frob}_p$ has order $2$ in $S_3$. If the Galois group is $\Z/3\Z$, or equivalently, if $\Delta_E$ is a square, the set of such primes $p$ has a density exactly $\frac{1}{3}$. However, if the Galois group is $S_3$, the set of such primes $p$ has a density exactly $\frac{2}{3}$. With this observation, we prove the following.
\begin{proposition}\label{prop:density}
    Let $E/\Q$ be an elliptic curve with trivial $E(\Q)[2]$. Then there exists a set of prime $S_E$ of positive density $d_E$ such that, for any squarefree integer $D$ whose all prime factors are from $S_E$, we have
    $$ \nu_2(m_{E^{(D)}})\geq 3\omega(D)+O_E(1).$$
    Here $d_E=\frac{1}{3}$, if $\Delta_E$ is a square, and $\frac{2}{3}$ otherwise. Moreover, we have  $ \nu_2(m_{E^{(D)}})\geq \omega(D)+O_E(1)$, for any squarefree integer $D$.
\end{proposition}
\begin{proof}
The proof of the second part follows immediately from (\ref{eqn:lbdtwist}). For the first part, denote $S_E$ as the set of all primes for which $a(p)$ is even. Viewing $\mathrm{Gal}(\Q(E[2])/\Q)$ inside $S_3$, asking $a(p)$ to be even is equivalent to asking for $\mathrm{Frob}_p$ not to have order $3$. When the Galois group is $S_3$, or equivalently, if $\Delta_E$ is a square, the required density is $\frac{1}{3}$. Similarly in the $S_3$ extension case, the density is $\frac{2}{3}$.
\end{proof}

To deduce a consequence of Watkins's conjecture, we study the rank of $E^{(D)}$ on an average. It is a conjecture of Goldfeld \cite{goldfeld79} that the average rank is $\frac{1}{2}$ for any elliptic curve $E/\Q$. Heath-brown \cite{HB2004} shows that BSD and GRH for the L-function of elliptic curves imply that the average is at most $\frac{3}{2}$. With this in hand, we have the following result.

\begin{theorem}\label{thm:twisttrivial}
  Let $E/\Q$ be any elliptic curve with trivial $E(\Q)[2]$. Then Watkins's conjecture is true for almost all quadratic twists, provided that BSD and GRH are true for the L-function of elliptic curves.
\end{theorem}

\begin{proof}
Denote $S$ as the set of squarefree integers $D$ for which Watkins's conjecture is not true for $E^{(D)}$. It is enough to prove that $S$ has density $0$. For the sake of contradiction, let us assume that the (lim-sup) density is $c>0$. It follows from Proposition~\ref{prop:density} that, $\mathrm{rank}(E^{(D)})\gg_{E}\omega(D)$, for any square-free integer $D$. In particular, it follows from Heath-Brown's result in \cite{HB2004} that
\begin{equation}\label{eqn:lbdsum}
x\gg \sum_{\substack{|D|\leq x,\\ D\in S}} \mathrm{rank}(E^{(D)})\gg_{E} \sum_{\substack{|D|\leq x,\\ D\in S}} \omega(D).
\end{equation} 
Denote $S_k(x)=\#\{|D|\leq x : D\in S,\omega(|D|)=k\}$, and note that $\sum_{k\leq \log x}S_k(x)\geq \#\{|D|\leq x : D\in S\}=cx$. Note that it is enough to only consider the contributions up to $k=\log x$, as $\omega(n)\leq \log n$. By the partial summation,
$$\sum_{\substack{|D|\leq x,\\ D\in S}}\omega(|D|)=\sum_{k\leq \log x}kS_{k}(x)\geq cx \log x+O\left((\log x)^2\right),$$
which is a contradiction to (\ref{eqn:lbdsum}).
\end{proof}
In this regard, we also would like to point out that if $E[2](\Q)$ is trivial, and $\nu_2(m_E)=r$, then $\omega(N)\leq r$ provided that $E$ has non-split multiplication reduction at the prime dividing $N$. In that case, root number of $E/\Q$ is $-1$. If $r=1$, then $N$ is a prime power. In the prime case, note that Watkins's conjecture follows from Kazalicki-Kohen in \cite{KK17singular}. On the other hand, if $E$ has a split multiplicative reduction, $\omega(N)$ is either $1$ or $2$. Again, for the case $\omega(N)=1$, the conjecture is proved in \cite{KK17singular} when $\Delta_E>0$. This leaves us with another possibility: what happens to the conjecture when $N=pq, E$ does not have split multiplicative at least at one prime and $E[2](\Q)$ is trivial.

\subsection{On a weaker question}
Let $E/\mathbb{Q}$ be any elliptic curve. This section discusses a much weaker question than Watkins's conjecture. 
\begin{question}
How often is  $m_E$  always even?
\end{question}
First of all, it follows from Lemma~\ref{lem:highomega} that $\nu_2(m_E)\geq 1$ for almost all elliptic curves over $\Q$ parameterized by $A,B\in \Z$. We now discuss the analogous distribution over any $ 1$-parameter family of elliptic curves. First, in the case of quadratic twists, it follows from Proposition~\ref{prop:density} that $\nu_2(m_{E^{(D)}})\gg_{E} \omega(D)$. In particular, $m_{E^{(D)}}$ is even for almost all $D$. Now we ask the same for any one-parameter family of elliptic curves given by
$$E_t:=y^2=x^3+f(t)x+g(t),~t\in \Z,$$
where $f(x), g(x) \in \Q[x]$ are two arbitrary polynomials. Then we have the following.
\begin{theorem}
    Let $E/\Q$ be any elliptic curve, then 
    $$\#\{t\in \Z : |t|\leq x,~ \nu_2(m_{E_t})\geq 1\}=O\left(\frac{x}{
    \log x
    }\right).$$  
\end{theorem}
The proof of this theorem follows from an analogous version of Lemma~\ref{lem:highomega}, which we briefly state now.
\begin{lemma}
    For at most $O\left(\frac{x}{\log x}\right)$ many $|t|\leq x$, we have $\omega(N_{E_t})\geq 1$.
\end{lemma}
To prove this lemma, we follow a similar path as in the proof of Lemma~\ref{lem:highomega}. Consider the polynomial $\Delta(x)=\Delta_{E_x}\in \Q[x]$, and then it is enough to show that $\#\{t\in \Z\mid |t|\leq x,\omega(\Delta(t)
)=1\}=O\left(\frac{x}{\log x}\right)$. For the Sieving argument, we consider $\{\mathcal{P}_{\lambda \vdash \mathrm{deg}(\Delta) }\}$, where $\mathrm{deg}(\Delta)$ denotes degree of the polynomial $\Delta$, and $\mathcal{P}_{\lambda}$ be the set of all primes $p$ for which $\Delta(x)\pmod p$ has a factorization of type $\lambda$ in $\mathbb{F}_p[x]$. Note that $\Delta(t)=0\pmod p$ for some prime $p$ if and only if, $p\in \mathcal{P}_{\lambda}$ for some partition $\lambda \vdash \mathrm{deg}(\Delta)$ such that $\lambda$ has at least one part which is $1$. Now for each such $\lambda$, we sieve modulo $\mathcal{P}_{\lambda}$.

Note that we got a significant advantage while working with the cases when $N_E$ has at least two prime factors. Now the question is, what happens for the families of elliptic curves whose conductors have only one prime factor? When $N$ is a prime, we know that $m_E$ is even when $\mathrm{rank}(E(\Q))>0$. For the rank $0$ case, we have the following result.
\begin{theorem}
    Let $E/\Q$ be any elliptic curve of prime conductor $p$. Then $m_E\# E_{\mathrm{tor}}$ is even, provided that $p \pmod {12} \in \{1,5,7\}$. If $p=11 \pmod {12}$, then $m_E\# E_{\mathrm{tor}}$ is even provided that Conjecture~\ref{conj:supsing} is true.
    \end{theorem}

\begin{proof}
  We recall the discussion in Section~\ref{sec:gross}. In the case $p\pmod {12}\in \{1,5\}$, we have all $w_i=1$ for all $i$, and hence, 
    $$m_E\# E_{tor}=\langle v_E,v_E\rangle =\sum_{i=1}^{n} v_E(e_i)^2=\sum_{i=1}^{n} v_E(e_i)=0 \pmod 2,$$
    where the last equality follows from \cite[Proposition 2.3]{KK18watkins}. In the case $p=7 \pmod{12}$, we have $w_i=3$ for exactly one $i$. In this case, a similar argument applies. In the case $p=11 \pmod {12}$, we have $w_i=2$ for exactly one $i$. In this case, $j(e_i)=1728\in \mathbb{F}_p$. The proof then follows from Conjecture~\ref{conj:supsing}.
    \end{proof}    
Note that $m_E$ is not necessarily even when $\#E_{\mathrm{tor}}$ is even, in other words, when $E(\Q)[2]$ is non-trivial. In that case, according to Setzer's classification, $p=17$ or of the form $u^2+64$. In the case, $p=17$, four non-isomorphic elliptic curves have levels respectively $17.a_1,17.a_2,17.a_3$ and $17.a_4$. It turns out that $m_E$ is even in the case of $17.a_1,17.a_2$ and $17.a_4$. For primes of the form $p=u^2+64$, two non-isomorphic elliptic curves of conductor $p$ have level $p.a_1$ and $p.a_2$. In the case of $p.a_1$, the discriminant $p$ is positive and not a square. Then it follows from \cite[Theorem 4]{KK17singular} and arguing similarly as in the proof of \cite[Theorem 1.4]{KK18watkins} we get $4|m_E\# E_{tor}$.

\subsection{On the function field analogue}\label{sec:wff}
Let $p>3$ be a prime and $k$ be a finite extension of $\mathbb{F}_p$. We define $A$ as the polynomial ring $k[T]$ and its field of fractions as $K = k(T)$. We further define $K_{\infty}$ as the completion of $K$ at $T^{-1}$, and $\mathbb{C}_{\infty}$ as the algebraic closure of $K_{\infty}$.

The Drinfeld upper half plane is denoted by $\Omega$, which is defined as $\mathbb{C}_{\infty} \setminus K_{\infty}$. It is important to note that the general linear group $\GL(2, K_\infty)$ acts on $\Omega$ through fractional linear transformations. Specifically, this action is performed by the Hecke congruence subgroup associated with an ideal $\mathfrak{n}$ of $A$, denoted by $\Gamma_0(\mathfrak{n})$. This subgroup is defined as follows:
\begin{align*}
	\Gamma_0(\mathfrak{n}) = \left\{
		g = \begin{pmatrix}
			a & b \\ c & d
		\end{pmatrix} \Big\vert \ a, b, c, d \in \mathbb{F}_q[T], \ c \equiv 0 \pmod{\mathfrak{n}}
		\right\}
\end{align*}
The quotient space $\Gamma_0(\mathfrak{n}) \setminus \Omega$ is compactified by adding the finitely many cusps $\Gamma_0(\mathfrak{n}) \setminus \mathbb{P}^1(K)$, resulting in the Drinfeld modular curve denoted by $X_0(\mathfrak{n})$.

Let $E$ be an elliptic curve over $K$ of conductor $\mathfrak{n}_{E}=\mathfrak{n}\cdot \infty $ and it has multiplicative reduction at $\infty$. Then we have a result analogous to the modularity result, which asserts that parametrization map from the Drinfeld modular curve $X_0(\mathfrak{n})$ to the elliptic curve $E$,  $\phi_E: X_0(\mathfrak{n}) \to E$ such that  $\phi_E$ is non-trivial and of minimal possible degree \cite[Theorem 2.1]{caro2022watkins}. The degree of this map is defined to be the modular degree of $E$.

Let $E$ be an elliptic curve with conductor $\mathfrak{n}_E = \mathfrak{n}\infty$, where $\mathfrak{n}$ is an ideal. Let $f_E$ be the primitive newform associated with $E$.
We denote by $\cW(\mathfrak{n})$ the 2-elementary
abelian group of all Atkin-Lehner involutions. Let $f$ be a primitive newform; since $f$ is primitive,
it is determined by its eigenvalues up to sign. We define $\cW' = \left \{W \in \cW(\fn) : W(f_E) = f_E\right\}$, 
which consists of the Hecke operators $W$ that preserve the newform $f_E$. Additionally,
let $\kappa :=\dim_{\mathbb{F}_2} \left( \cW(\fn) / \cW' \right) + \dim_{\mathbb{F}_2}E(K)[2]$, where  $E(K)[2]$ represents the $2$-torsion points of $E$ defined over $K$. Then, by \cite[Proposition 3.2]{caro2022watkins}, we have the inequality $\omega_K(\fn) - \kappa \leq \nu_2(m_E)$, where $\omega_K(n)$ denotes the number of prime divisors of $\fn$ in the field $K$, and $\nu_2(m_E)$ represents the $2$-adic valuation of the  modular degree $m_E$ associated with $E$. This is proved by realizing $M(\Gamma_0(\mathfrak{n}),\Q_{\ell})$ as the dual of $V_{\ell}(J_0(\mathfrak(n)))$. To be more precise, this identification shows that $\pi([W(D)])=\pi([D])$ for every divisor $D$ of degree $0$ over $X_0(\mathfrak{n})$. The proof follows a similar path as that of \cite[Proposition 2.1]{DK13}, as it assumes that the characteristic of $k$, denoted by $p$, is odd.

In this functional field situation, we have the following rank bound given by Tate \cite{Tate95},
\begin{equation}\label{eqn:rankbddf}
\mathrm{rank}_{\Z}(E(K))\leq \mathrm{deg}(\mathfrak{n}_E)-4,
\end{equation}
where $\mathfrak{n}_E$ is the finite part of the conductor of $E/K$. By utilizing the lower bound of $\nu_2(m_E)$, Caro \cite{caro2022watkins} \textit{potentially} establishes the validity of the Watkins conjecture. In other words, denote $k'$ as the splitting field of $\mathfrak{n}$ over the finite field $k$, and set $K' = k'(T)$. Then Watkins's conjecture is true for the base change over $K'$, i.e., for the elliptic curve $E' = E \times_{\text{Spec } K} \text{Spec } K'$, the Watkins's conjecture is true. Here, $E$ represents a modular semi-stable elliptic curve defined over $K$ with conductor $\mathfrak{n}E = \mathfrak{n}{\infty}$. 
where $\mathfrak{n}$ is the conductor and $\mathfrak{n}_{\infty}$ is a finite field containing the splitting field of $\mathfrak{n}$ over $k$. Caro needed the base over $K'$ to ensure that $\omega_{K'}(\mathfrak{n})=\mathrm{deg}(\mathfrak{n})$, and the helps the lower bound of $\nu_2(m_E)$ to beat the rank bound at (\ref{eqn:rankbddf}). However, we can use the same technique to deduce the following, allowing us to work with a possibly smaller $K'$.
\begin{theorem}\label{thm:watkinsf}
    Let $E/K$ be any elliptic curve of conductor $\mathfrak{n}=\mathfrak{n}_E\infty$. Let $k''$ be the extension of $k$ over which $\mathfrak{n}$ has the property that 
    $$\deg(\mathfrak{n})-\omega_{K''}(\mathfrak{n}_E)\leq 3-\kappa.$$
    Then Watkins's conjecture is true for $E\times_{\mathrm{Spec}(K'')}\mathrm{Spec}(K'')$.
\end{theorem}
We omit the proof but note that $\kappa$ is always at most $3$; in that case, we have no improvement. However, $\kappa$ could still be $0,1$ or $2$, for instance, $\kappa=0$ if and only if, $E(K)[2]$ is trivial and $E$ has split-multiplicative reduction at every prime. In all these cases of $\kappa$, we have a possibly smaller extension of $K$ contained in $K'$ that obeys the imposed condition at Theorem~\ref{thm:watkinsf}.

In the same article, Caro discusses Watkins's conjecture for certain quadratic twists, i.e., twists by polynomials of even degree. The assumption on degree is needed to ensure that the twisted elliptic curve is modular, as remarked by Caro~\cite[Section 4]{caro2022watkins}. With this, Caro showed that Watkins's conjecture is true for any such even twist, provided that $\omega_K(g)\geq 3$ and $E/K$ are semi-stable. On the other hand, if $E/K$ is semi-stable but with non-trivial $E(K)[2]$, then similarly, as in the case of $\Q$, we also have the following function field analog of the rank bound
 \begin{equation}\label{eqn:rankbddft}
 \mathrm{rank}_{\Z}(E(K))\leq \omega_K(\mathfrak{n}).
 \end{equation} 
With this, Caro showed that any $g$ of even degree works, provided that $E(K)[2]=\Z/2\Z$ and $E$ have non-split multiplicative reduction everywhere. By a similar argument, one can show that if, instead, we have $E(K)[2]=(\Z/2\Z)^2$, then a similar conclusion holds if $\omega_{K}(g)\geq 2$.
\section{On the growth of modular degree}\label{sec:bounds}
In this section, we shall discuss the growth properties of $m_E$ in both $\Q$ and functional field situations. Let us first consider $E/\Q$ to be an elliptic curve of conductor $N$. It is conjectured that $m_E=O(N^{2+\varepsilon})$. Now recall the two identities from (\ref{eqn:height}). It follows from the proof of \cite[Theorem 1]{Murty99} that $N^{1-\varepsilon}<||f_E||^2_N<N^{1+\varepsilon}$. From \cite[Lemma 2.1]{watkins2004explicit} we have an explicit lower bound of the form $m_E> \frac{60\pi^2N^{1-\varepsilon}}{D_E^{1/6}}$. Watkins \cite{watkins2004explicit} showed that $L(\mathrm{Sym}^2(E),1)\gg \frac{1}{N^{(2)}}$, by showing that there exists no real zero for $L(\mathrm{Sym}^2(E),s)$ in a region of the form $\mathrm{Re}(s)\geq 1-\frac{C}{\log N^{(2)}}$, where $N^{(2)}$ is conductor of the associated L-function, which is $\geq N_E^2$. This gives a lower bound on $m_{E}$ of order $\gg N^{7/6-\varepsilon}$. Here $L(\mathrm{Sym}^2(E),s)$ is the L-function associated with the motive $H^1(\mathrm{Sym}^2(E))$, whose local factors $L_p(s)$, are of the polynomials in $p^{-s}$, that are determined by $\mathrm{det}(1-\mathrm{Frob}_p|{\mathrm{Sym}^2}(T_{\ell}(E)\otimes \Q_{\ell}))$, for any prime $p\neq \ell$. The existence of the zero-free region essentially shows that $m_E\gg \frac{N^{7/6}}{\log N}\prod_{p\mid N}L_p(1)$. The local factors of the (motivic) Symmetric square $L$-function is easier to understand at the primes of good and multiplicative reductions. For the other primes, these factors are $\geq 1$ (see \cite[page 5, Corollary]{watkins2004explicit}), whenever $p\equiv 1 \pmod {12}$, and hence the lower bound follows.

One approach to get an upper bound would be to know the bounds for Manin's constant and Falting's height of $E$. Indeed, for any integer $A,B$, consider the \textit{Frey} curve $E: y^2=x(x-A)(x+B)$. It turns out that the Frey curves are semistable away from $2$. In particular, $c_{E}$ is uniformly over this family. Then a computation of $j$ and $\gamma$ of $E_{A, B, C}$, gives the connection with the ABC problem plugging them in (\ref{eqn:height}). In particular, it turns out that ABC implies the conjectural degree bound holds. In fact, any unconditional result on ABC contributes an unconditional bound on the modular degree. For instance, \cite{SY01} implies a bound of order $e^{O_{\varepsilon}(N^{1/3+\varepsilon})}$.

Murty in \cite{Murty99} considered the conjecture for elliptic curves with CM. Other than the ones with $j$ invariants, respectively $0$ and $1728$, it turns out that any elliptic curve with CM is the quadratic twist of a set of finitely many elliptic curves $\mathcal{E}$. Note that $j_{E}=j_{E^{(D)}}$ for any $E\in \mathcal{E}$, and in particular, $h(E)-h(E^{(D)})=\frac{1}{2}\log D$. This shows that 
\begin{equation}\label{eqn:mebound}
(1-\varepsilon)\log{N_D}+\log D\ll_{E}\log m_{E^{(D)}}\ll_{E} (1+\varepsilon)\log{N_D}+\log D,
\end{equation}
where $N_D$ is the conductor of $E^{(D)}$, which is $N_ED^2$ if $(N_E,D)=1$, and otherwise, $N_D=O(N_ED^2)$. This shows a lower bound of the form $m_E\gg N_{E}^{3/2+\varepsilon}$ for any elliptic curve $E$ with CM. Also, as long as the upper bound is concerned, the conjectural bound holds whenever the twist by $D$ satisfies that $(d,N_E)=1$. 

We can now prove the following by adapting to the arguments in the last two paragraphs.

\begin{theorem}\label{thm:degreeconj}
For a positive proportion of elliptic curves in the family of all elliptic curves (resp. in the family of twists of a given $E$) satisfies the conjectural degree bound. 
\end{theorem}

\begin{proof}
For any set of finitely many primes $S$, denote $\mathcal{E}_{S}$ be the set of all elliptic curves over $E/\Q$ for which $N_E$ is semistable away from $S$. Note that $c_{E}=O_S(1)$ uniformly over this family. It is known due to Fouvry et al. \cite{FNT92} that $\log(\Delta_E)\leq (1+\varepsilon) \log N_E$ for almost all elliptic curves over $E/\Q$, when arranged by height. On the other hand for any $\delta>0$, the number of pairs $(A,B)\in \Z^2$, with $\left|\frac{A^3}{\Delta_{E}}\right|<1+\delta$ has density at most $\delta$. Choosing $\delta$ suitably, we now have that $h(j_E)=\log (\Delta_{E})\leq (1+\varepsilon)\log N_E$, for all elliptic curves outside a set of density $O(\varepsilon)$. We now claim that, $\gamma=O_S(1)$, for any $E\in \mathcal{E}$. Modulo this, we have 
$m_E\ll_{S}N_E^{7/6+\varepsilon}$, for any $E\in \mathcal{E}_{S}$. It is then enough to prove the claim and show that $\mathcal{E}_S$ has a density at least $1-\left(\sum_{p \not\in S }\frac{1}{p^2}\right)$. 

First, to prove the claim, note that $E$ has an additive reduction for any prime $p$ dividing $\gamma$. Since $E\in \mathcal{E}_p$, it is evident that $p\in S$. It is enough to bound the power of $p$ by dividing $\gamma$. For this, we need to chase Kodaira's classification. For instance, from Table 1 in \cite{SSW21}, we immediately see that the power is bounded for all the types, other than $I_n,I^{*}_n~(n\geq 1)$. Following the notations of \cite[Theorem 1.6]{SSW21}, note $j_E=\frac{16a^2-48b}{\Delta}$, and the $p$-adic valuation of numerator of at most $6$, for any prime $p$ and for any $n\geq 2$. This shows that the valuation of $\Delta_E$ in $j_E$ does not go down too much, and in particular, this analysis shows that $\nu_p(\gamma)\leq 6$ for any prime $p\in S$. 

Now to compute the density of $\mathcal{E}_S$, again note from Table 1 in \cite{SSW21} that, the proportion of all elliptic curves having additive reduction at a prime $p$ is $1/p^2$. This shows that, the density of $\mathcal{E}_{S}$ is at least $1-\left(\sum_{p \not\in S }\frac{1}{p^2}\right)>\frac{1}{2}+\sum_{p\in S}\frac{1}{p^2}>0$, as desired. 

Now let $E/\Q$ be any given elliptic curve, and $\{E^{(D)}\}$ be the family of quadratic twists. For any $D$ co-prime to $N_E$, we have $N_{E^{(D)}}=ND^2$. Then it follows from (\ref{eqn:mebound}) that 
$$N_{E^{(D)}}^{3/2-\varepsilon}\ll m_{E^{(D)}}\ll N_{E^{(D)}}^{3/2+\varepsilon},$$
for any such $D$. This completes the proof. 
\end{proof}
\begin{remark}\rm
    It follows from the proof of Theorem~\ref{thm:degreeconj} that a stronger estimate $m_E\ll N_E^{7/6+\varepsilon}$. Assuming Manin's conjecture, i.e., that Manin's constant is uniformly bounded, the same estimate with exponent $7/6+\varepsilon$ holds for almost all elliptic curves. Moreover, the reader also may note from the proof of Theorem~\ref{thm:degreeconj} that exponent $3/2+\varepsilon$ can be achieved for a positive proportion of elliptic curves in any family of quadratic twists.  
\end{remark}

Note that the degree bound could also be seen as one of the consequences for the bound of $||f_E||^2$. Let us now quickly mention another consequence. Rankin in \cite[Theorem 1]{Rankin39} showed that $\sum_{1\leq n\leq x}|a_E(n)|^2=\alpha_E x^2+o(x^2),$
where $\alpha_E>0$ is some computable constant. In particular, applying partial summation formula, it follows from \cite{Murty99} that $\alpha_E=O\left(\frac{\phi(N)||f_E||^2}{N^2}\right)=O(\log N)$. We would also like to point out that an inexplicit version of the square-sum problem for any Fuchsian group (of the first kind) is discussed in \cite{BBF22}.
 
Regarding the degree bound conjecture, one may also ask for a weaker problem; is any prime factor $\ell_E$ of $m_E$ is $O(N^{2+\varepsilon})$? Now one may ask about the growth of $\ell_E$, varying over all the elliptic curves $E/\Q$. Note that any such $\ell_E$ is also a congruence prime for $E$, i.e., there exists a cuspform $g(z)=\sum_{n>0}a_g(n)q^n$ with integer Fourier coefficients $b_n$, such that $a_E(n) \equiv a_g(n) \pmod{\ell_E}$ for all $n$. Suppose that $g$ corresponds to an elliptic curve $E'/\Q$. By Brauer–Nesbitt theorem, the representation $\rho_{E,E'}:\mathrm{Gal}_{\Q}\to \Delta(\ell_{E})$ is not surjective, when $\ell_E>2$; see \cite{Jones10} for the definition of $\Delta(\ell_E)$ and this representation. With this observation, we can prove the following.

\begin{theorem}\label{thm:congavg}
    Let $E/\Q$ be any elliptic curve, and $\ell_E\geq 5$ be a prime diving $m_E$. Denote $\mathcal{E}(x)$ be the set of all elliptic curves over $\Q$ of height at most $x$. Then, $g$ corresponds to at most $O(x^{4}\log^{2} x)$ many elliptic curves in $\mathcal{E}(x)$.
\end{theorem}
To prove this, we follow the similar Sieving approach as in the proof of (10) in \cite{Jones13}. The difference is that, we have to carry out the sieving on a copy of $\mathbb{Z}^2$ inside $\mathbb{Z}^4$.
\begin{proof}[Proof of Theorem~\ref{thm:congavg}]

For each prime $p$, and each conjugacy class $C$ in $\mathrm{SL}_2(\Z/\ell_E\Z)$, set 
    $$\Omega_{E}(C)=\{(r,s)\in (\Z/p\Z)^2: \pi_2(\rho_{E,E_{r,s}}(\mathrm{Frob}(p)))\in C\},$$
    where $\pi_2$ denotes the natural projection onto the second component. Regardless of what the first component $E$ is, $\Omega_{E}(C)$ coincides with $\{(r,s)\in (\Z/p\Z)^2: \rho_{E_{r,s}}(\mathrm{Frob}(p))\in C\}$, provided that $p=1\pmod {\ell_E}$. In that case, we have from \cite[Theorem 8]{Jones10} that $\# \Omega_{E}(C)=p^2\left(\frac{\# C}{\# \mathrm{SL}_2(\Z/\ell_E \Z)}+O\left(\frac{\# C}{p^{1/2}}\right)\right)$. This shows that, 
    $$P_{C}(x)=\sum_{p\leq x}\frac{\Omega_{E}(C)}{p^2}=\frac{\# C}{\# \mathrm{SL}_2(\Z/\ell_E \Z)}\pi(x,\ell_E,1)+O(\# Cx^{1/2}),$$
    where $\pi(x,\ell_E,1)$ denotes $\{p\leq x: p=1\pmod {\ell_E}\}$. Let us now denote $\mathcal{E}_{E}(x)$ be the set of all elliptic curves $E'$ in $\mathcal{E}(x)$, that corresponds to $g$. Due to \cite[Lemma 5]{duke97}, except for at most $x^{4}\log^{2} x$ many $E'$ in $\mathcal{E}(x)$, we can assume that $\mathrm{im}(\rho_{E'})=\mathrm{GL}_2(\Z/\ell_E\Z)$. Again by Brauer-Nesbitt, we have $\mathrm{im}(\rho_{E})=\mathrm{GL}_2(\Z/\ell_E\Z)$. Since $\mathrm{im}(\rho_{E,E'})\neq \Delta(\ell_E)$ for any such $E'$, by \cite[Lemma 3.1]{Jones13} we know that, there exists conjugacy classes $C_1,C_2$ in $\mathrm{SL}_2(\Z/\ell_E\Z)$ such that $\mathrm{im}(\rho_{E,E'})\cap C_1\times C_2=\phi$. By large sieve, we argue similarly as in \cite[page 3389]{Jones13}, and get for any conjugacy class $C$ that
    $$ \pi(x,\ell_E,1)^2\# \mathcal{E}_{E}(x)\leq \sum_{E'\in \mathcal{E}(x)}(\pi_{E,E'}(x,C)-\delta_C \pi(x,\ell_E,1))^2\ll \# \mathcal{E}(x) P_C(x),$$
    where $\delta_C=\frac{\# C}{\# \mathrm{SL}_2(\Z/\ell_E \Z)}$, and $\pi_{E,E'}(x,C)=\{p\leq x : \pi_2(\rho_{E,E'}(\mathrm{Frob}_p))\in C\}$. Hence, 
    $\# \mathcal{E}_{E}(x)\ll \sum_{C_1\times C_2}\frac{ \# \mathcal{E}(x) P_C(x)}{\pi(x,\ell_E,1)^2}\ll_{\ell_E} x^{4}\log^2 x$, as desired.
\end{proof}

\begin{remark}\rm
    Of course, $g$ is not necessarily an eigenform. Then we can write, $g$ as linear combination of a family of hecke eigenforms $\{f_i\}_{s(N)}$, where $s(N)=\mathrm{dim}(S_2(\Gamma_0(N)))$. We can then consider the family of representations $\{\rho_{E,E_1,E_2,\cdots, E_{s_2(N)}}\}_{\substack{E_i\in \mathcal{E}(x)\\ 1\leq i\leq s(N)}}$, and draw a similar conclusion as in Theorem~\ref{thm:congavg}, applying the multi dimensional analog of Large sieve.
\end{remark}
\subsection{Function field analog of growth}\label{sec:ffgrowth} Following the set-up used in Section~\ref{sec:wff}, we denote $|\mathfrak{n}|_{\infty}=q^{\mathrm{deg}(\mathfrak{n})}$, and $\mathrm{deg}_{\mathrm{ns}}(j_E)$ be the nonseparable degree of $\mathbb{F}_q(T)/\mathbb{F}_q(j_E)$. Given any newform $f_E$, it turns out that $L(Sym^2f_E,2)=q\frac{||f_E||^2}{\mathfrak{n}}$. Moreover, $\mathfrak{m}_E=\frac{||f_E||^2}{\mathrm{val}_{\infty}(j_E)}$, where $\mathrm{val}_{\infty}(j_E)$ denotes the number of irreducible components of the singular fiber $\infty$ (note that, we always assume that $E$ has bad reduction at $\infty$). Of course $1\leq \mathrm{val}_{\infty}(j_E)\leq \mathrm{deg}(\Delta_E)$, and hence, $\frac{||f_E||^2}{ \mathrm{deg}(\Delta_E)}\leq \mathfrak{m}_E\leq ||f_E||^2$. Moreover, it follows from \cite{PS2000} that $\mathrm{deg}(\Delta_E)\leq \mathrm{deg}_{\mathrm{ns}}(j_E)\mathrm{deg}(\mathfrak{n}_E)$, where $\mathrm{deg}_{\mathrm{ns}}(j_E)$ denotes the non-seperable degree of $\mathbb{F}_q(T)/\mathbb{F}_q(j_E)$. Papikian \cite{PM2002} provided both analytic and algebraic ways to bound the norm of $f_E$. In the analytic approach, the bound follows as a consequence of Ramanujan's conjecture on the holomorphic $L$-functions on $\mathbb{C}$, and Rademacher's version of Phragmen-Lindelof's theorem for the estimation of $L$-functions on a strip. Now for the algebraic approach, Grothendieck's theory $L$-function implies that $L(s,\mathrm{Sym}^2(E))$ can be written as a quotient of two polynomials in $q^{-s}$, whose degree is bounded by $2\mathrm{deg}(\mathfrak{n}_E)-4$. This is a consequence of Grothendieck–Ogg–Shafarevich's formula, realizing $\mathrm{Sym}^2$ as a constructible sheaf over $\mathbb{P}^1_{\mathbb{F}_q}$. To get the degree bound, they assumed that $E$ is semi-stable, which guarantees that all the wild parts $\delta_{\mathfrak{p}}|\mathfrak{p}\in \mathbb{P}^1_{\mathbb{F}_q}$ are $0$. However, a weaker bound of the form $\ll \mathrm{deg}(\mathfrak{n}_E)$ (on the degree of $L(s,\mathrm{Sym}^2(E))$), saves the day for us. For this, we may put a weaker (than being semistable) condition on $E$, that sum of all the wild parts, i.e., $\sum_{\mathfrak{p}\in \mathbb{P}^1_{\mathbb{F}_q}}\delta_{\mathfrak{p}}$ is $O(\mathfrak{n}_E)$. Conversely, a lower bound is also achieved using some standard analytic tools. To summarize the discussion, we now can state the bounds \cite[Corollary 6.2]{PM2002} in a nice geometrical way; $$\log \left(\frac{\mathfrak{n}_E}{\mathrm{deg}_{\mathrm{ns}}(j_E)}\right)\ll \log (m_{E})\ll \log (\mathfrak{n}_E).$$

Now the question that naturally comes to our mind is that, whether it is possible to remove the extra dependency on $E$, that appears as a factor $\mathrm{deg}_{\mathrm{ns}}(j_E)$ in the lower bound. As pointed out by Papikian in \cite[page 347]{PM2002} that, on the family $\{\Delta_{E^{q^n}}\}_{n\geq 1}\}$, discriminant increases, but the conductor remains same, as $n$ grows. Hence, the extra factor can not be removed for all elliptic curves. Following the spirit as in all the previous sections, we consider the family of all elliptic curves, and quadratic twists, to study the proportion of elliptic curves where it is possible to remove the factor. First, given any elliptic curve $E:= y^2=x^3+f(T)x+g(T)$ with $f(T),g(T)\in \mathbb{F}_q[T]$, we define height $H(E)=\max\left\{3\mathrm{deg}(f),2\mathrm{deg}(g)\right\}$. This is the function field analog of height in the $\Q$ case. Denote $\mathcal{E}_{\mathbb{F}}(d)$ be the number of elliptic curves of height at most $d$, which is precisely $q^{5d/6}$. Now note that 
$$\mathrm{deg}_{\mathrm{ns}}(j_E)\leq \mathbb{F}_q(T)/\mathbb{F}_q(j_E)=\max\left\{3\mathrm{deg}(f),2\mathrm{deg}(g)\right\}.$$
This shows that for any integer $D\leq d$, we have the bound $m_E\gg \frac{\mathrm{deg}(\mathfrak{n}_E)}{D}$, for least $q^{5D/6}$ many $E$ in $\mathcal{E}_{\mathbb{F}}(d)$. On the other hand, as long as we consider the families of quadratic twists, we see that $j_E$ remains invariant. Therefore, in any such family, the factor $\mathrm{deg}_{\mathrm{ns}}(j_E)$ remains unchanged.

\section*{Acknowledgements}
The work for this project took place at IISER, TVM. We would like to thank the institute for providing excellent working conditions. 
\bibliographystyle{amsplain} 
	\bibliography{ref.bib}
\end{document}